\documentclass{amsart}
\usepackage[T1]{fontenc}
\usepackage{latexsym,amssymb,amsmath,mathrsfs}

\usepackage{amsthm}
\usepackage{longtable}
\usepackage{epsfig}
\usepackage{hhline}
\usepackage{epic}

\catcode`\Ž=13
\defŽ{\'e}
\catcode`\ˆ=13
\defˆ{\`a}
\catcode`\=13
\def{\`e}
\catcode`\‰=13
\def‰{\^a}
\catcode`\=13
\def{\c c}
\catcode`\™=13
\def™{\^o}
\catcode`\š=13
\defš{\'E}
\catcode`\Š=13
\defŠ{\`A}
\catcode`\Ÿ=13
\defŸ{\^o}
\catcode`\=13
\def{\^e}
\catcode`\"=13
\def"{\^\i}
\catcode`\=13
\def{\`u}
\catcode`\•=13
\def•{\"\i}

   \newcommand{\cyc}[1]{\langle\,#1\,\rangle}
   
   \newcommand{\Irr}{\operatorname{Irr}}
   
   \newcommand{\sym}{\mathfrak{S}}

   \newcommand{\A}{\mathfrak A}
   
   \newcommand{\la}{\lambda}
    \newcommand{\laOw}{\lambda_0^{(w)}}
    \newcommand{\laiw}{\lambda_i^{(w)}}
        \newcommand{\lajw}{\lambda_j^{(w)}}
     \newcommand{\laOwm}{\lambda_0^{(w-1)}}
    \newcommand{\laiwm}{\lambda_i^{(w-1)}}

 \newcommand{\modp}{ \; (\mbox{mod} \; p)}

  \newcommand{\ba}{\beta}
  \newcommand{\sa}{\sigma}

 \newcommand{\ga}{\gamma}

  \newcommand{\p}{\bar{p}}
  
\newcommand{\bH}{\overline{H}}
\newcommand{\cal}[1]{\mathcal{#1}}

  \newcommand{\floor}[1]{\lfloor #1 \rfloor}

\long\def\symbolfootnote[#1]#2{\begingroup\def\thefootnote{\fnsymbol{footnote}}
\footnote[#1]{#2}\endgroup}

\newtheorem{theorem}{Theorem}[section] 
\newtheorem{corollary}[theorem]{Corollary}

\newtheorem{proposition}[theorem]{Proposition}

\title[Malle-Navarro Conjecture for covering groups]{On a conjecture of G. Malle and G. Navarro on nilpotent blocks}

\address{Institut de Math\'ematiques de Jussieu\\
Universit\'e Denis Diderot, Paris VII\\
UFR de Math\'ematiques\\
2 place Jussieu\\
F-75251 Paris Cedex 05\\
email: gramain@math.jussieu.fr}

\begin{document}
\maketitle

\begin{center}
{\Large{\bf{Jean-Baptiste Gramain}}}\\

\medskip

Institut de MathŽmatiques de Jussieu\\
UniversitŽ Denis Diderot, Paris VII\\
UFR de MathŽmatiques\\
2 place Jussieu\\
F-75251 Paris Cedex 05\\
email: gramain@math.jussieu.fr
\end{center}

\begin{abstract}
In a recent article, G. Malle and G. Navarro conjectured that the $p$-blocks of a finite group all of whose height 0 characters have the same degree are exactly the nilpotent blocks defined by M. BrouŽ and L. Puig. In this paper, we check that this conjecture holds for spin blocks of the covering group $2.\A_n$ of the alternating group $\A_n$, thereby solving a case excluded from the study of quasi-simple groups by Malle and Navarro.
\end{abstract}

\symbolfootnote[0]{2000 Mathematics Subject Classification 20C30 (primary), 20C15, 20C20 (secondary)}

\symbolfootnote[0]{Keywords: Representation Theory, Nilpotent Blocks, Symmetric Group, Covering Groups, Bar-Partitions}

\section{Introduction}\label{intro}

In a recent paper (\cite{Malle-Navarro}), G. Malle and G. Navarro have formulated a conjecture about nilpotent blocks of finite groups. The notion of nilpotent block was first introduced by M. BrouŽ and L. Puig in \cite{Broue-Puig}, and should be the most natural to study from a local point of view. However, the definition given by BrouŽ and Puig uses the Alperin-BrouŽ subpairs, making the detection of nilpotent blocks a difficult problem. One strong property of nilpotent blocks is that, if a $p$-block $B$ of a finite group $G$ is nilpotent, then all the height zero characters $\chi \in \Irr_0(B)$ have the same degree. In \cite{Malle-Navarro}, Malle and Navarro conjecture that the converse also holds, therefore giving a global characterization of nilpotent blocks which is visible in the character table of $G$.

\medskip
In their paper, Malle and Navarro prove that their conjecture is true whenever $B$ is the principal block of $G$ (\cite[Theorem 3.1]{Malle-Navarro}), or if the defect group $D$ of $B$ is normal in $G$ (\cite[Theorem 5.2]{Malle-Navarro}). They also prove that it holds whenever $D$ is abelian, provided Brauer's Height Zero Conjecture holds (\cite[Theorem 4.1]{Malle-Navarro}), and make considerable progress in the case of $p$-solvable groups. Finally, they give a proof of their conjecture for all finite quasi-simple groups (\cite[Theorem 6.1]{Malle-Navarro}), with the possible exception of quasi-isolated blocks of exceptional groups of Lie type in bad characteristic, and faithful blocks of the 2-fold covering group $2. \A_n$ of the alternating group $\A_n$ ($n \geq 14$).

\medskip
The objective of this paper is to prove that the covering group $2.\A_n$ does not in fact yield any counter-example to the conjecture of Malle and Navarro, i.e. that any block of $2.\A_n$ all of whose height zero characters have the same degree is nilpotent (Corollary \ref{finalresult}). In Section \ref{covering}, we introduce the classical results about characters and blocks we need to study the case of $2.\A_n$. In Section \ref{barpartitions}, we construct, for blocks of $2. \sym_n$ with non-abelian defect group, height zero characters with distinct degrees. Finally, Section \ref{restriction} is devoted to restricting these characters to $2.\A_n$ and checking that they do provide the desired result in this case.

Note that, even though our method is analogous to that used by Malle and Navarro in the case of $\A_n$, the fact that we use bar partitions and bars instead of partitions and hooks induces several complications. Also, there is no clear bar analogue of the relative hook-formula for character degrees they use in the symmetric group.

\section{Characters and blocks of covering groups}\label{covering}
In this section, we present an overview of the representation theory of the covering groups of $\sym_n$ and $\A_n$. These groups were first introduced and studied by I. Schur in \cite{Schur}. Unless stated otherwise, all the results in this section (and references for proofs) can be found for example in \cite{Olsson-Combinatorics}.

The symmetric group $\sym_n$ has, for $n \geq 4$, two non-isomorphic 2-fold covering groups (only one if $n=6$) which are isoclinic, and therefore have virtually identical representation theory. We therefore denote, slightly abusively, by $2.\sym_n$ one of these covering groups. Then $2.\sym_n$ has center $\cyc{z}$ of order 2, and $2.\sym_n/\cyc{z} \cong \sym_n$. The group $2. \sym_n$ has a (unique, normal) subgroup of index 2, which is the unique 2-fold covering group $2.\A_n$ of the alternating group $\A_n$. 

\smallskip
The irreducible complex characters of $2.\sym_n$ and $2.\A_n$ fall into two categories. If $\chi \in \Irr( 2.\sym_n)$ (or $\chi \in \Irr(2.\A_n)$), and if $z \in \ker(\chi)$, then $\chi$ is just lifted from an irreducible character of $\sym_n$ (or $\A_n$). Otherwise, $\chi$ is a faithful character, also called {\emph{spin character}}, and corresponds to a projective representation of $\sym_n$ or $\A_n$.

\smallskip
The spin characters of $2.\sym_n$ and $2.\A_n$ are canonically labelled by the {\emph{bar partitions}} of $n$, i.e. partitions of $n$ in distinct parts. If $\la=(a_1 > \cdots > a_m >0)$ is a bar partition of $n$, then we let $m(\la)=m$ and $\sa(\la)=(-1)^{n-m(\la)}$. If $\sa(\la)=1$, then $\la$ labels a unique spin character of $2.\sym_n$. The restriction to $2.\A_n$ of this character is the sum of two {\emph{associate}} spin characters, both labelled by $\la$. If, on the other hand, $\sa(\la)=-1$, then $\la$ labels two associate spin characters of $2.\A_n$; both have the same restriction to $2.\A_n$, which is the unique spin character of $2.\A_n$ labelled by $\la$.

\smallskip
The notion that replaces that of hooks in partitions is given by {\emph{bars}} in bar partitions. If $\la=(a_1 > \cdots > a_m >0)$ is a bar partition of $n$, then the set of {\emph{bar lengths}} in $\la$ is$$\overline{{\cal H}}(\la)=\displaystyle \bigcup_{1 \leq i \leq m} \{ 1, \, \ldots , \, a_i \} \cup \{ a_i + a_j \, | \, j>i\} \setminus \{ a_i - a_j \, | \, j >i \} .$$
For any integer $\ell$, we call {\emph{$\bar{\ell}$-weight}} of $\la$ the number of bars in $\la$ whose length is divisible by $\ell$. Such a bar is called an $(\ell)$-bar.

Writing $\bH(\la)$ for the product of all the bar lengths in $\la$, we have, in analogy with the Hook-Length Formula for the degree of characters of $\sym_n$, that any spin character of $2.\sym_n$ labelled by $\la$ has degree $2^{\floor{(n-m(\la))/2}} \frac{n!}{\bH(\la)}$.

\medskip
If we now take a prime $p$, then the distribution of irreducible characters of $2.\sym_n$ and $2.\A_n$ into $p$-blocks depends on the parity of $p$.

If $p$ is odd, then every $p$-block $B$ of $2.\sym_n$ or $2. \A_n$ contains either no spin character, or only spin characters. A block with no spin character is really just a block of $\sym_n$ or $\A_n$, and its defect groups are the same as in these groups. In particular, the conjecture of Malle and Navarro holds for these blocks of $2.\A_n$ because it holds in $\A_n$ (by \cite[Corollary 9.3]{Malle-Navarro}). If $B$ contains only spin characters, then $B$ is refered to as a {\emph{spin block}}, or {\emph{faithful block}}. Such a block has either defect 0, hence contains a unique spin character (labelled by a bar partition with $\p$-weight 0), or it consists exactly of all the spin characters labelled by bar partitions with a given {\emph{$\p$-core}} (the bar partition obtained by removing from a bar partition all the bars of length divisible by $p$). This is known as the Morris Conjecture, which is the analogue of the Nakayama Conjecture for the symmetric group, and was proved in \cite{Humphreys-Blocks} and \cite{Cabanes}. Each spin block $B$ of $2.\sym_n$ of $\p$-weight $w \geq 1$ covers a unique block $B^*$ of $2.\A_n$, which is labelled by the same $\p$-core (and each spin block of $2.\A_n$ is covered by a unique spin block of $2.\sym_n$). The defect groups of $B$ and $B^*$ are entirely determined by $w$, and are the same as the defect groups of any block of $p$-weight $w$ in $\sym_n$. We also see from the degree formula that a spin character of $B$ has height 0 if and only if it's labelled by a bar partition of ($\p$-weight $w$ and) maximal $\p^2$-weight, maximal $\p^3$-weight, etc. Finally, note that Brauer's Height Zero Conjecture is known to be true in this case (see Olsson \cite{Olsson-blocks}). To check that Malle and Navarro's conjecture holds in this case, it is thus sufficient to check that, if such a block has non-abelian defect groups, then it contains two non-spin characters of height 0 with distinct degrees.

\smallskip
If, on the other hand, $p=2$, then each $p$-block of positive defect of $2.\sym_n$ or $2.\A_n$ contains both spin characters and non-spin characters (for a complete description of the 2-block structure of $2.\sym_n$ and $2.\A_n$, see \cite{Bessenrodt-Olsson-2blocks}).

By \cite[Corollary 9.3]{Malle-Navarro}, any such block $B$ of $2.\A_n$ all of whose height zero non-spin characters have the same degree corresponds to a block $b$ of $\A_n$ with the same property, and thus of defect 0 in $\A_n$. Any such $B$ therefore has cyclic defect group of order 2 in $2.\A_n$. Since Brauer's Height Zero Conjecture holds in this case, $B$ is nilpotent. Hence the conjecture of Malle and Navarro is true for $2.\A_n$ when $p=2$ (see Corollary \ref{finalresult}).

\medskip
From now on, we therefore always suppose that $p$ is an odd prime, and only consider spin blocks.

\section{Bar-partitions with different products of bar lengths}\label{barpartitions}

Take any odd prime $p$, and let $\ga$ be a $\p$-core.
Until the last result of this section, we assume furthermore that $\ga \neq \emptyset$.

We write $\ga=(a_1 > \cdots > a_m >0)$ and $X_{\ga}=\{ a_1, \, \ldots , \, a_m \}$. The bars in $\ga$ correspond to
\begin{itemize}

\item
pairs $(x, \, y)$, with $0 \leq x < y$, $x \not \in X_{\ga}$ and $y \in X_{\ga}$ (these have length $y-x$, and type 2 if $x=0$ and type 1 if $x \neq 0$), and

\item
pairs $(i, \, j)$, with $1 \leq i < j \leq m$ (these have length $a_i + a_j$ and type 3).
\end{itemize}

We start by giving three easy consequences of the fact that $\ga$ is a $\p$-core. Firstly, for all $1 \leq i \leq m$, we have $a_i \not \equiv 0 \modp$ (otherwise, $\ga$ would have a $(p)$-bar of type 2). We can thus arrange the $a_i$'s according to their value mod $p$, and let
$$X_0= \{ a_i \, | \, a_i  \equiv 0 \modp \} = \emptyset \; \mbox{and} \; X_j= \{ a_i \, | \, a_i  \equiv j \modp \}  \; \mbox{for} \; 1 \leq j \leq p-1.$$If $X_j \neq \emptyset$, then $X_{p-j}=\emptyset$ (otherwise, we would have $a_t=pk_t+j \in X_j$ and $a_s=pk_s+(p-j) \in X_{p-j}$, and thus the $(p)$-bar $a_t+a_s=p(k_t+k_s+1)$ of type 3). Finally, if $X_j \neq \emptyset$, then $X_j=\{ j + kp \, , \, 0 \leq k \leq d_j \}$. Indeed:
\begin{itemize}
\item
$j \in X_j$ since, otherwise, $(j, \, a_i)$ (for any $a_i \in X_j$) would give a $(p)$-bar of type 1 (note that, obviously, if $j \in X_{\ga}$, then $j \in X_j$);
\item
if $j + sp \in X_j$ for some $s>0$, then $j + tp \in X_j$ for all $0 \leq t < s$ (since, otherwise, $(j+tp, \, j + sp)$ would give a $(p)$-bar of type 1).
\end{itemize}
If $X_j = \emptyset$, then we set $d_j = -1$. For any $0 \leq j \leq p-1$, we also write $e_j=j+d_j p$.



\medskip
Now take any integer $w \geq 1$. We define some bar partitions with $\p$-core $\ga$ and $\p$-weight $w$.
\begin{itemize}
\item
Define $\laOw$ by letting $X_{\laOw}=X_{\ga} \cup \{ pw \}= \bigcup_{j=0}^{p-1} X_j^{0,(w)}$ (so that $X_0^{0,(w)}=\{pw\}$ and $X_j^{0,(w)}=X_j$ for $1 \leq j \leq p-1$). Note that $\laOw$ can be defined even if $\ga = \emptyset$.
\item
If $X_i \neq \emptyset$ (for some $1 \leq i \leq p-1$), then define $\laiw$ by letting $X_{\laiw}=X_{\ga} \cup \{ e_i  + pw \} \setminus \{ e_i \}= \bigcup_{j=0}^{p-1} X_j^{i,(w)}$ (so that $X_0^{i,(w)}=\emptyset $, $X_j^{i,(w)}=X_j$ for $1 \leq j \neq i \leq p-1$, and $X_i^{i,(w)}=X_i \cup \{ i + d_i p + pw \} \setminus \{i + d_i p\}$).

\end{itemize}

Let also $\la_0^{(0)}=\la_i^{(0)}=\ga$.

Note that $\laOw$ and (if $X_i \neq \emptyset$) $\laiw$ are bar partitions of $|\ga|+pw$, with $\p$-weight {\bf{at least}} $w$ and $\p$-core $\ga$, and thus $\p$-weight {\bf{exactly}} $w$.

Note also that, by definition, all the $(p)$-bars in $\laOw$ are of type 1 or of type 2 (and there's exactly one of these), while, if $X_i \neq \emptyset$, then all the $(p)$-bars in $\laiw$ are of type 1.

Finally, note that $m(\laOw)=m(\ga)+1$, while, if $X_i \neq \emptyset$, then $m(\laiw)=m(\ga)$.

\medskip
First, we suppose that $X_i \neq \emptyset$, and we want to compare $\bH(\laiw)$ and $\bH(\laiwm)$ (the products of all bar lengths in $\laiw$ and $\laiwm$ respectively). Note that, if $w=1$, then $\laiwm=\ga$. Write $\bH(\laiw)=\bH_m(\laiw)\bH_u(\laiw)$ and $\bH(\laiwm)=\bH_m(\laiwm)\bH_u(\laiwm)$, where $\bH_m$ (respectively $\bH_u$) stands for the product of all mixed (respectively unmixed) bar lengths, i.e. of type 3 (respectively of type 1 or 2).

\begin{proposition}\label{barsi}
With the above notation, we have
$$\displaystyle \frac{\bH_u(\laiw)}{\bH_u(\laiwm)} = pw \prod_{0 \leq j \neq i \leq p-1} | p(w-1) + e_i - e_j | $$and$$\displaystyle \frac{\bH_m(\laiw)}{\bH_m(\laiwm)} = \prod_{j \neq i, \, X_j \neq \emptyset} \frac{e_i+e_j+pw}{e_i+p(w-1)+j}.$$
\end{proposition}

\begin{proof}
To prove the first part, one can simply compare explicitly all the unmixed bars in $\laiw$ and $\laiwm$. Another way is to apply \cite[Theorem 9.1]{Malle-Navarro} (or the special case that immediately follows) to the partitions $(\laiw)^*$ and $(\laiwm)^*$, which correspond to the $\ba$-sets  $X_{\laiw}$ and  $X_{\laiwm}$. Indeed, these have $p$-core $\ga^*$ (the partition corresponding to the $\beta$-set $X_{\ga}$), $p$-weight $w$ and $w-1$ respectively, and their hooks and hook-lengths correspond exactly to the unmixed bars and their lengths in $\laiw$ and $\laiwm$.

\medskip

We now turn to the second part. The only mixed bars which are not common to $\laiw$ and $\laiwm$ have lengths:
\begin{itemize}
\item
in $\laiw$: $\{ i+d_i p +pw +x_k^{(j)} \, | \,  j \neq i, \, X_j \neq \emptyset , \, 0 \leq k \leq d_j \}$, where $x_k^{(j)}=j+kp$;
\item
in $\laiwm$: $\{ i+d_i p +p(w-1) +x_k^{(j)} \, | \,  j \neq i, \, X_j \neq \emptyset , \, 0 \leq k \leq d_j \}$.
\end{itemize}
Since $i+d_i p +pw +x_k^{(j)}=i+d_i p +p(w-1) +x_{k+1}^{(j)}$, when we divide out, we're only left with $k=d_j$ in $\laiw$ and $k=0$ in $\laiwm$. This yields
$$\displaystyle \frac{\bH_m(\laiw)}{\bH_m(\laiwm)} =\prod_{j \neq i, \, X_j \neq \emptyset} \frac{i +d_ip+pw+j +d_jp}{i+d_ip+p(w-1)+j}= \prod_{j \neq i, \, X_j \neq \emptyset} \frac{e_i+e_j+pw}{e_i+p(w-1)+j}.$$

\end{proof}
\begin{corollary}\label{barstwo}
If $X_i \neq \emptyset$, then
$$ \frac{\bH(\laiw)}{\bH(\laiwm)} = pw (p(w-1) + e_i + i) \prod_{j \neq i \atop X_j \neq \emptyset} | p(w-1) + e_i - e_j |   (pw+e_i+e_j) \prod_{X_k = \emptyset \atop X_{p-k} = \emptyset} | pw+e_i-k | .$$
\end{corollary}
\begin{proof}
From Proposition \ref{barsi}, we easily obtain, if $X_i \neq \emptyset$,
$$\displaystyle \frac{\bH(\laiw)}{\bH(\laiwm)} = pw \prod_{0 \leq j \neq i \leq p-1} | p(w-1) + e_i - e_j |  \; \; \prod_{j \neq i,  \, X_j \neq \emptyset} \frac{e_i+e_j+pw}{e_i+p(w-1)+j}.$$
Separating the first product according to whether $X_j = \emptyset$ or not, we get
$$pw \prod_{ j \neq i \atop X_j \neq \emptyset} | p(w-1) + e_i - e_j | (pw + e_i + e_j)   \prod_{k \neq i \atop X_k = \emptyset} |p(w-1) + e_i - e_k | \prod_{j \neq i  \atop X_j \neq \emptyset} \frac{1}{p(w-1)+e_i+j}.$$
Now, if $X_j \neq \emptyset$, then $X_{p-j}= \emptyset$ (since $\ga$ is a $\p$-core), so that $p-j \neq i$ (since $X_i \neq \emptyset$), and $j=-[(p-j)-p]=-e_{p-j}$. We thus have
$$ \prod_{k \neq i,  \, X_k = \emptyset} |p(w-1) + e_i - e_k | \prod_{j \neq i,  \, X_j \neq \emptyset} \frac{1}{p(w-1)+e_i+j}$$
$$
= \prod_{k \neq i,  \, X_k = \emptyset} |p(w-1) + e_i - e_k | \prod_{j \neq i, \, p-j \neq i,  \, X_j \neq \emptyset, \, X_{p-j}=\emptyset} \frac{1}{p(w-1)+e_i-e_{p-j}}$$
$$
=  (p(w-1) + e_i - e_{p-i})\prod_{k \neq i,  \, X_k = \emptyset, \atop X_{p-k} \neq \emptyset} \frac{ |p(w-1) + e_i - e_k |}{p(w-1)+e_i-e_k} \prod_{k \neq i,  \atop X_k = X_{p-k}=\emptyset} | p(w-1)+e_i-e_k |.$$
And, since $X_i \neq \emptyset$, we have $e_{p-i}=-i$ and $p(w-1) + e_i - e_{p-i}=p(w-1) + e_i + i$.

If $X_k = \emptyset$, then $e_k=k-p<0$, so that $p(w-1)+e_i-e_k>0$ (since $e_i \geq 0$), whence the first product is $1$, and $p(w-1)+e_i-e_k=pw+e_i-k$. Finally, if $X_k = \emptyset$, then $k \neq i$. We therefore get, if $X_i \neq \emptyset$,
$$\displaystyle \frac{\bH(\laiw)}{\bH(\laiwm)} = pw (p(w-1) + e_i + i) \prod_{j \neq i, \atop X_j \neq \emptyset} | p(w-1) + e_i - e_j |   (pw+e_i+e_j) \prod_{X_k = \emptyset , \atop X_{p-k} = \emptyset} | pw+e_i-k | .$$

\end{proof}

{\bf{Remark:}} If $X_i \neq \emptyset$, then $e_i \geq i$, so that $e_i-k \geq i-k >-p$ and, whenever $w \geq 1$, we have $| pw + e_i - k| = pw+e_i-k$.

\medskip

We now establish the analogous result for the bar partitions $\laOw$ and $\laOwm$:

\begin{proposition}
With the above notation, we have
\begin{itemize}
\item
if $w>1$, then
$$\displaystyle \frac{\bH_u(\laOw)}{\bH_u(\laOwm)} = pw \prod_{0< j  \leq p-1} | p(w-1) - e_j | $$and$$ \displaystyle \frac{\bH_m(\laOw)}{\bH_m(\laOwm)} = \prod_{j \neq 0, \, X_j \neq \emptyset} \frac{e_j+pw}{p(w-1)+j};$$
\item
if $w=1$, then
$$\displaystyle \frac{\bH_u(\laOw)}{\bH_u(\laOwm)} = p \frac{\prod_{j \neq 0} | e_j |}{\prod_{1 \leq i \leq m} a_i }$$and$$ \displaystyle \frac{\bH_m(\laOw)}{\bH_m(\laOwm)} = \prod_{j \neq 0, \, X_j \neq \emptyset} \; \;  \prod_{0 \leq k \leq d_j} (p+j+kp) = \prod_{1 \leq i \leq m} (a_i+p )     .$$
\end{itemize}
\end{proposition}

\begin{proof}
We start with the unmixed bars. All the unmixed bars $(x, \, y)$ in $\laOw$ such that $x \not \equiv 0 \modp$ and $y \not \equiv 0 \modp$ are also in $\laOwm$, and conversely. Hence we just need to consider the unmixed bars $(x, \, y)$ with $x \equiv 0 \modp$ or $y \equiv 0 \modp$.

Those with $x \equiv y \equiv 0 \modp$ contribute exactly $pw$ to $\frac{\bH_u(\laOw)}{\bH_u(\laOwm)}$.

Next suppose $y \equiv 0 \modp$, and $x \equiv j \modp$, with $0< j \leq p-1$. The bar lengths to consider are thus:
\begin{itemize}
\item
in $\laOw$: $\{ pw - x_k^{(j)} \, | \, j \neq 0, \, d_j \leq w-2, \, 1 \leq k \leq w-d_j-1 \}$, where $x_k^{(j)}=e_j+kp$;
\item
in $\laOwm$: $\{ p(w-1) - x_k^{(j)} \, | \, j \neq 0, \, d_j \leq w-3, \, 1 \leq k \leq w-d_j-2 \}$. These only appear if $w>1$; however, since, when $w=1$, $\{j \, | \, d_j \leq w-3\}=\emptyset$, the result applies in this case too.
\end{itemize}
Now, for any $j\neq 0$ such that $d_j \leq w-3$, we have $pw - x_k^{(j)}=p(w-1)-x_{k-1}^{(j)}$ for all $1 \leq k \leq w-d_j-1$, so that everything cancels out in $\frac{\bH_u(\laOw)}{\bH_u(\laOwm)}$, except for $k=1$. We're thus left exactly with $pw-x_1^{(j)}=p(w-1)-e_j$. If, on the other hand, $d_j=w-2$, then the only bar that appears is given by $k=1$ in $\laOw$, and it contributes $p(w-1)-e_j$ to $\frac{\bH_u(\laOw)}{\bH_u(\laOwm)}$.

\smallskip
Finally, suppose $x \equiv 0 \modp$, and $y \equiv j \modp$, with $0< j \leq p-1$. Now, for $X \in \{ X_{\laOw}, \, X_{\laOwm} \}$, we must have $x=pr \not \in X$, $y=j+kp \in X$, whence $k \leq d_j$, and $y>x$, so that $k \geq r$. Hence we just need to exclude $r=w$ in $\laOw$ (since $wp \in X_{\laOw}$), and $r=w-1$ in $\laOwm$ (since $(w-1)p \in X_{\laOwm}$), except if $w=1$ (in which case $x=(w-1)p=0 \not \in X_{\laOwm}=X_{\ga}$). We thus get, for the product of these bar lengths,
\begin{itemize}
\item
in $\laOw$: $\displaystyle \prod_{r \geq 0, \, r \neq w} \; \; \prod_{r \leq k \leq d_j \; (d_j \geq r)} (j+kp-pr)$;
\item
in $\laOwm$: $\displaystyle \prod_{r \geq 0, \, r \neq w-1 \, \mbox{\tiny{if}} \, w>1} \; \; \prod_{r \leq k \leq d_j \; (d_j \geq r)} (j+kp-pr)$.
\end{itemize}
After cancellations (corresponding to fixed $r \not \in \{w, \, w-1 \}$), we're left with
\begin{itemize}
\item
in $\laOw$: $\displaystyle \prod_{w-1 \leq k \leq d_j  \; (d_j \geq w-1)} (j+kp-p(w-1))$;
\item
in $\laOwm$: $\displaystyle \prod_{w \leq k \leq d_j  \; (d_j \geq w)} (j+kp-pw)$, 

and also, if $w=1$, $\displaystyle \prod_{w-1 \leq k \leq d_j \; (d_j \geq w-1)} (j+kp-p(w-1))$.
\end{itemize}
After further cancellations, we see that, whether $d_j \geq w$ or $d_j=w-1$, we're only left with $k=d_j$ in $\laOw$, which contributes $j+d_jp-p(w-1)=e_j-p(w-1)$ to $\frac{\bH_u(\laOw)}{\bH_u(\laOwm)}$. In addition to this, and only in the case $w=1$, we have, left in $\laOwm$, $\displaystyle \prod_{w-1 \leq k \leq d_j \; (d_j \geq w-1)} (j+kp-p(w-1))=\prod_{0 \leq k \leq d_j \; (d_j \geq 0)} (j + kp) = \prod_{1 \leq i \leq m} a_i$.

\smallskip
Putting together the three cases for the values of $x$ and $y$, we obtain the announced expressions for $\frac{\bH_u(\laOw)}{\bH_u(\laOwm)}$.

\medskip
Turning now to mixed bars, we see that the only ones which are not common to $\laOw$ and $\laOwm$ have lengths

\begin{itemize}
\item
in $\laOw$: $\{ pw + x_k^{(j)} \, | \, j \neq 0, \, X_j \neq \emptyset, \, 0 \leq k \leq d_j \}$, where $x_k^{(j)}=j+kp$;
\item
in $\laOwm$, only if $w>1$: $\{ p(w-1) + x_k^{(j)} \, | \, j \neq 0, \, X_j \neq \emptyset, \, 0 \leq k \leq d_j \}$.
\end{itemize}
If $w>1$, then $pw+x_k^{(j)}=p(w-1)+x_{k+1}^{(j)}$, so that we're only left with $k=d_j$ in $\laOw$ and $k=0$ in $\laOwm$. This yields
$$ \displaystyle \frac{\bH_m(\laOw)}{\bH_m(\laOwm)} = \prod_{j \neq 0, \, X_j \neq \emptyset} \frac{pw+j+d_jp}{p(w-1)+j}.$$
If $w=1$, then we obtain$$ \displaystyle \frac{\bH_m(\laOw)}{\bH_m(\laOwm)} = \prod_{j \neq 0, \, X_j \neq \emptyset} \; \;  \prod_{0 \leq k \leq d_j} (p+j+kp) = \prod_{1 \leq i \leq m} (a_i+p ),$$
as announced.

\end{proof}

\begin{corollary}\label{barszero}
$$\displaystyle \frac{\bH(\laOw)}{\bH(\laOwm)} =\left\{ \begin{array}{ll}  \displaystyle pw \prod_{ j \neq 0} | p(w-1) - e_j |  \; \; \prod_{j \neq 0} \frac{pw+e_j}{p(w-1)+j}  & \mbox{if} \; w>1,  \\ \displaystyle  p \prod_{j \neq 0} |e_j| \prod_{1 \leq i \leq m} \frac{a_i + p}{a_i} & \mbox{if} \; w=1. \end{array} \right. $$
\end{corollary}

We can now show that it's always possible to construct two bar partitions with same bar core $\ga \neq \emptyset$ and weight $w \geq 1$ such that the corresponding products of bar lengths are distinct.

\begin{theorem}\label{bars}
Suppose there exist $i_1 \neq i_2$ such that $i_1, \, i_2 >0$, $X_{i_1} \neq \emptyset$ and $X_{i_2} \neq \emptyset$, and suppose $e_{i_1} > e_{i_2} > e_j$ for all $j \not \in \{i_1, \, i_2 \}$. Then, whenever $w \geq 1$, we have $\bH(\la_{i_1}^{(w)})>\bH(\la_{i_2}^{(w)})$. If, on the other hand, there exists a unique $i \neq 0$ such that $X_i \neq \emptyset$, then, whenever $w \geq 1$, we have $\bH(\laiw)>\bH(\laOw)$.

\end{theorem}

\begin{proof}

First suppose $i_1, \, i_2 >0$, $X_{i_1} \neq \emptyset$, $X_{i_2} \neq \emptyset$, and $e_{i_1} > e_{i_2} > e_j$ for all $j \not \in \{i_1, \, i_2 \}$. We use Corollary \ref{barstwo}.

Whenever $X_k=X_{p-k}=\emptyset$, we have $k \not \in \{i_1, \, i_2 \}$, and $pw+e_{i_1}-k > pw + e_{i_2}-k>0$ as soon as $w \geq 1$. Thus, unless both products are empty,
$$ \displaystyle \prod_{X_k = X_{p-k} = \emptyset} ( pw+e_{i_1}-k ) >  \prod_{X_k = X_{p-k} = \emptyset} ( pw+e_{i_2}-k ).$$
If, on the other hand, $j \not \in \{i_1, \, i_2 \}$ and $X_j \neq \emptyset$, then $0 < e_j < e_{i_2} < e_{i_1}$. Hence $e_{i_1}+e_j, \, e_{i_2}+e_j, \, e_{i_1}-e_j, \, e_{i_2}-e_j >0$. Thus $0<pw+e_{i_2}+e_j<pw+e_{i_1}+e_j$, and
$$|p(w-1)+e_{i_2}-e_j|=p(w-1)+e_{i_2}-e_j< p(w-1)+e_{i_1}-e_j=|p(w-1)+e_{i_1}-e_j|,$$whence
$$|p(w-1)+e_{i_1}-e_j| (pw+e_{i_1}+e_j) > |p(w-1)+e_{i_2}-e_j| (pw+e_{i_2}+e_j).$$
Finally, $e_{i_1}-e_{i_2}>0$, so that $|p(w-1)+e_{i_1}-e_{i_2}| = p(w-1)+e_{i_1}-e_{i_2}$ and
$$|p(w-1)+e_{i_1}-e_{i_2}| (pw+e_{i_1}+e_{i_2}) \geq |p(w-1)+e_{i_2}-e_{i_1}| (pw+e_{i_2}+e_{i_1}),$$and this last inequality is in fact strict unless $w=1$, in which case the spin block we consider has abelian defect.

We thus obtain
$$\displaystyle \prod_{j \neq i_1, \, X_j \neq \emptyset} | p(w-1) + e_{i_1} - e_j |   (pw+e_{i_1}+e_j) > 
 \prod_{j \neq i_2, \, X_j \neq \emptyset} | p(w-1) + e_{i_2} - e_j |   (pw+e_{i_2}+e_j) .$$
 
 Finally, for each $k \in \{1, \, 2 \}$, we have $e_{i_k}=i_k+pd_{i_k}$. Now, since $e_{i_1} > e_{i_2}$, we have either $d_{i_1}>d_{i_2} \geq 0$, or $d_{i_1}=d_{i_2}$. If $d_{i_1}>d_{i_2} \geq 0$, then $pd_{i_1} \geq pd_{i_2}+p$, and $e_{i_1} + i_1 > e_{i_2} + i_2$ (since $0 < i_1, \, i_2 \leq p$). If, on the other hand, $d_{i_1}=d_{i_2}$, then $i_1 > i_2$, so that $e_{i_1} + i_1 > e_{i_2} + i_2$.
 
We thus have $p(w-1)+e_{i_1} + i_1 > p(w-1)+e_{i_2} + i_2$, whence, by Corollary \ref{barstwo}, $$\displaystyle \frac{\bH (\la_{i_1}^{(w)}) }{\bH (\la_{i_1}^{(w-1)}) }   >   \frac{\bH (\la_{i_2}^{(w)}) }{\bH (\la_{i_2}^{(w-1)}) } \; \; \; \mbox{whenever  } w \geq 1.$$
Since $\la_{i_1}^{(0)} = \la_{i_2}^{(0)} = \ga$, induction on $w$ gives that $\bH(\la_{i_1}^{(w)})>\bH(\la_{i_2}^{(w)})$ if $w \geq 1$.

\smallskip
Suppose now that there exists a unique $i \neq 0$ such that $X_i \neq \emptyset$. Then $X=X_i= \{ i+kp, \, 0 \leq k \leq d_i \}= \{ a_m, \, \ldots , \, a_1\}=\{ i , \, i+p, \, \ldots , \, e_i \}$. Also, $e_i=i+d_ip \geq i$, and $e_j=j-p$ whenever $j \neq i$. By Proposition \ref{barsi}, we have, whenever $w \geq 1$,
$$\begin{array}{rl} \displaystyle \frac{\bH(\laiw)}{\bH(\laiwm)} & \displaystyle = pw \prod_{0 \leq j \neq i \leq p-1} | p(w-1) + e_i - (j-p) | \; \; \mbox{(the other product being 1)}  \\ & \displaystyle = pw \prod_{0 \leq j \neq i \leq p-1} | pw + e_i - j | = pw \prod_{0 \leq j \neq i \leq p-1} ( pw + e_i - j ). \end{array}$$
On the other hand, by Corollary \ref{barszero}, we have, if $w=1$,
$$\begin{array}{rl} \displaystyle \frac{\bH(\laOw)}{\bH(\laOwm)}  & = \displaystyle  p \prod_{j \neq 0} |e_j| \prod_{1 \leq k \leq m} \frac{a_k + p}{a_k}  = \displaystyle  p \left( \prod_{j \neq 0} |j-p| \right) \frac{e_i+p}{i} \\ & = \displaystyle  p \left( \prod_{ 1 \leq j \leq p-1 } (p-j) \right) \frac{e_i+p}{i} = p (p-1)! \displaystyle  \frac{e_i+p}{i} = p! \displaystyle  \frac{e_i+p}{i}, \end{array}$$
while, if $w>1$, then we have
$$\begin{array}{rl} \displaystyle \frac{\bH(\laOw)}{\bH(\laOwm)}  & = \displaystyle pw \prod_{ j \neq 0} | p(w-1) - e_j  | \; \; \mbox{(the other product being 1)} \\  & = \displaystyle pw | p(w-1) - e_i  | 
 \prod_{ 1 \leq j \neq i \leq p-1} | p(w-1) - (j-p)  |   \\ & = \displaystyle pw | p(w-1) - e_i  | 
 \prod_{ 1 \leq j \neq i \leq p-1} (pw - j) . \end{array} $$
If $w>1$, then, whenever $1 \leq j \neq i \leq p-1$, we have $pw-j<pw+e_i-j$, and (for $j=0$) $|p(w-1)-e_i| < p(w-1) + e_i < pw+e_i-0$. Hence, in this case,
$$\displaystyle \frac{\bH(\laiw)}{\bH(\laiwm)} > \frac{\bH(\laOw)}{\bH(\laOwm)}.$$
If $w=1$, then
$$\displaystyle \frac{\bH(\la_i^{(1)})}{\bH(\la_i^{(0)})}=p \prod_{ 0 \leq j \neq i \leq p-1} (p+e_i-j)=p (p+e_i) \prod_{ 1 \leq j \neq i \leq p-1} (p+e_i-j).$$
Now $\prod_{ 1 \leq j \neq i \leq p-1} (p+e_i-j) \geq \prod_{ 1 \leq j \neq i \leq p-1} (p+i-j)$ (since $e_i \geq i >0$). We rewrite $ \prod_{ 1 \leq j \neq i \leq p-1} (p+i-j)$ as $  \prod_{ 1 \leq j  \leq i-1} (p+i-j)\prod_{ i+1 \leq j  \leq p-1} (p+i-j)$. Then $  \prod_{ 1 \leq j  \leq i-1} (p+i-j)= (p+1)(p+2) \cdots (p+i-1)> 1.2 \cdots (i-1)=(i-1)!$, and $\prod_{ i+1 \leq j  \leq p-1} (p+i-j) =(p-1)(p-2) \cdots (p+i-(p-1))=\frac{(p-1)!}{i!}$. Hence $\prod_{ 1 \leq j \neq i \leq p-1} (p+e_i-j)> \frac{(i-1)!(p-1)!}{i!}=\frac{(p-1)!}{i}$ and$$\displaystyle \frac{\bH(\la_i^{(1)})}{\bH(\la_i^{(0)})}>p(p+e_i)\frac{(p-1)!}{i}=p!\frac{p+e_i}{i}= \frac{\bH(\la_0^{(1)})}{\bH(\la_0^{(0)})}.$$
Since $\la_i^{(0)}=\la_0^{(0)}=\ga$, induction on $w$ yields that $\bH(\laiw)>\bH(\laOw)$ whenever $w \geq 1$.

\end{proof}

Finally, we deal with the case of the principal spin blocks, that is the spin blocks of $2.\sym_n$ and $2.\A_n$ labelled by the empty bar core.

\begin{proposition}\label{gavide}

For any $w \geq 2$ and odd prime $p$, the bar partitions $\mu_0^{(w)}=(pw)$ and $\mu_1^{(w)}=(pw-1, \, 1)$ of $pw$ satisfy $\bH(\mu_0^{(w)}) > 2 \bH(\mu_1^{(w)})$.

\end{proposition}

\begin{proof}

This is obvious, since the bar lengths in $\mu_0^{(w)}$ are $\{ 1, \, 2, \, \ldots , \, pw-1, \, pw\}$, while those in $\mu_1^{(w)}$ are $\{ 1, \, 2, \, \ldots , \, pw-3, \, pw-1, \, pw, \, 1\}$, and, since $w \geq 2$ and $p \geq 3$, we have $pw-2 >2$.

\end{proof}

\section{Height zero spin characters of $2.\sym_n$ and $2.\A_n$}\label{restriction}
We can now prove our main result

\begin{theorem}\label{mainresult}
Let $n\geq 4$ be any integer and $p$ be an odd prime. If $B$ is a spin $p$-block of $2.\A_n$ with non-abelian defect groups, then $B$ contains two height 0 characters which have distinct degrees.
\end{theorem}

\begin{proof}
Let $\ga$ be the $\p$-core labelling $B$ and the corresponding spin $p$-block $B^*$ of $2.\sym_n$. Let $w$ be the $\p$-weight of $B$ and $B^*$. Since $B$ has non-abelian defect, we have $w \geq p$. We use the notation of Section \ref{barpartitions}, and we write $\Irr_0(B)$ for the set of irreducible (spin) characters of height zero in $B$.

First suppose that $\ga = \emptyset$, so that $n=pw$. Take any two spin characters $\chi_0$ and $\chi_1$ of $2.\sym_{n}$ labelled respectively by the bar partitions $\mu_0^{(w)}=(pw)$ and $\mu_1^{(w)}=(pw-1, \, 1)$, and take spin characters $\psi_0$ and $\psi_1$ appearing in the restrictions to $2.\A_{n}$ of $\chi_0$ and $\chi_1$ respectively. In particular, we have $\psi_0, \, \psi_1 \in \Irr_0(B)$. We have

$$ \displaystyle  \chi_0(1)=2^{\floor{(n-1)/2)}}\frac{n!}{\bH(\mu_0^{(w)})} \; \; \mbox{and} \; \;  \chi_1(1)=2^{\floor{(n-2)/2)}}\frac{n!}{\bH(\mu_1^{(w)})} .$$

First suppose that $n=pw$ is odd. Then $\sa(\mu_0^{(w)})= 1$ and $\sa(\mu_1^{(w)})= -1$, so that, when restricted to $2.\A_n$, $\chi_0$ splits while $\chi_1$ doesn't. We thus have $\psi_0(1)=\frac{1}{2} \chi_0(1)$ and $\psi_1(1)=\chi_1(1)$. Since, in this case, $\floor{(n-2)/2)}=\floor{(n-1)/2)}-1$, we have, by Proposition \ref{gavide}
$$\frac{1}{2} \chi_0(1)=\frac{2^{\floor{(n-1)/2)}-1}n!}{\bH(\mu_0^{(w)})}=\frac{2^{\floor{(n-2)/2)}}n!}{\bH(\mu_0^{(w)})}<\frac{2^{\floor{(n-2)/2)}}n!}{\bH(\mu_1^{(w)})}=\chi_1(1),$$
so that $\psi_0, \, \psi_1 \in \Irr_0(B)$ satisfy $\psi_0(1)<\psi_1(1)$.

\smallskip
Suppose now that $n$ is even. In this case, when restricted to $2.\A_n$, $\chi_1$ splits while $\chi_0$ doesn't, and $\floor{(n-2)/2)}=\floor{(n-1)/2)}$. By Proposition \ref{gavide}, we get
$$\psi_0(1)=\chi_0(1)=\frac{2^{\floor{(n-1)/2)}}n!}{\bH(\mu_0^{(w)})}<\frac{2^{\floor{(n-2)/2)}}n!}{2\bH(\mu_1^{(w)})}=\frac{1}{2}\chi_1(1)=\psi_1(1),$$
so that $\psi_0, \, \psi_1 \in \Irr_0(B)$ satisfy $\psi_0(1)<\psi_1(1)$.

This proves that, whenever the principal spin $p$-block $B$ of $2.\A_n$ has weight $w \geq 2$ (in particular, when $B$ has non-abelian defect group), $B$ contains two height 0 characters with distinct degrees.

\medskip

From now on, we therefore suppose that $\ga \neq \emptyset$. If there exist $i \neq j$ such that $X_i \neq \emptyset$ and $X_j \neq \emptyset$, then we can suppose $e_i>e_j>e_k$ for all $k \not \in \{i, \, j\}$, and consider the bar partitions $\laiw$ and $\lajw$ of $n$ constructed in Section \ref{barpartitions}. They have the same number of parts $m(\laiw)=m(\lajw)=m(\ga)=m$. If $\chi_i$ and $\chi_j$ are spin characters of $2.\sym_n$ labelled by $\laiw$ and $\lajw$ respectively, then, by construction, $\chi_i$ and $\chi_j$ have $p$-height 0, and, by Theorem \ref{bars}
$$ \displaystyle  \chi_i(1)=2^{\floor{(n-m)/2)}}\frac{n!}{\bH(\laiw)} < 2^{\floor{(n-m)/2)}}\frac{n!}{\bH(\lajw)}= \chi_j(1).$$
Also, $\sa(\laiw)=\sa(\lajw)$, so that both $\chi_i$ and $\chi_j$ split when restricted to $2.\A_n$, or none of them does. In both cases, we thus obtain $\psi_i, \, \psi_j \in \Irr_0(B)$ which satisfy $\psi_i(1) < \psi_j(1)$.

\medskip
Now suppose there exists a unique $i$ such that $X_i \neq \emptyset$, and consider the bar partitions $\laiw$ and $\laOw$ of $n$. Then $m(\laiw)=m(\ga)=m$ and $m(\laOw)=m+1$. If $\chi_i$ and $\chi_0$ are spin characters of $2.\sym_n$ labelled by $\laiw$ and $\laOw$ respectively, then, by construction, $\chi_i$ and $\chi_0$ have $p$-height 0, and we have 
$$ \displaystyle  \chi_i(1)=2^{\floor{(n-m)/2)}}\frac{n!}{\bH(\laiw)} \; \; \mbox{and} \; \;  \chi_0(1)=2^{\floor{(n-m-1)/2)}}\frac{n!}{\bH(\laOw)} .$$

First suppose that $n-m$ is even, so that $\sa(\laiw)=1$ and $\sa(\laOw)=-1$. When restricted to $2.\A_n$, $\chi_i$ thus splits, while $\chi_0$ doesn't. Taking characters $\psi_i$ and $\psi_0$ in these restrictions, we obtain, by Theorem \ref{bars}, and since in this case $\floor{(n-m-1)/2)}=\floor{(n-m)/2)}-1$,
$$\frac{1}{2} \chi_i(1)=\frac{2^{\floor{(n-m)/2)}-1}n!}{\bH(\laiw)}=\frac{2^{\floor{(n-m-1)/2)}}n!}{\bH(\laiw)}<\frac{2^{\floor{(n-m-1)/2)}}n!}{\bH(\laOw)}=\chi_0(1),$$
so that $\psi_i, \, \psi_0 \in \Irr_0(B)$ satisfy $\psi_i(1)<\psi_0(1)$.
   
\smallskip
Suppose now that $n-m$ is odd, so that $\sa(\laiw)=-1$ and $\sa(\laOw)=1$. When restricted to $2.\A_n$, $\chi_0$ thus splits, while $\chi_i$ doesn't. Taking characters $\psi_0$ and $\psi_i$ in these restrictions, we have
$$\psi_i(1)= \chi_i(1)=2^{\floor{(n-m)/2)}}\frac{n!}{\bH(\laiw)},$$
and, since in this case $\floor{(n-m-1)/2)}=\floor{(n-m)/2)}$,
$$\psi_0(1)=\frac{1}{2} \chi_0(1)=\frac{2^{\floor{(n-m-1)/2)}}n!}{2\bH(\laOw)}=\frac{2^{\floor{(n-m)/2)}}n!}{2\bH(\laOw)}.$$
We therefore obtain $\psi_0(1) \neq \psi_i(1)$, unless $\bH(\laiw)=2\bH(\laOw)$.

\noindent
To exclude this last possibility, write, for each bar partition $\nu$, $\bH(\nu)=\bH_p(\nu)\bH_{p'}(\nu) $, where $\bH_p(\nu)$ (respectively $\bH_{p'}(\nu) $) is the $p$-part (respectively the $p'$-part) of $\bH(\nu)$. Since $\chi_i$ and $\chi_0$ both have $p$-height 0, we have $\bH_p(\laiw)=\bH_p(\laOw)$. Also, since all the $(p)$-bars in $\laiw$ and $\laOw$ are of type 1 or of type 2, Proposition 2.5 in \cite{JB-IsaacsNavarro} gives $\bH_{p'}(\laiw) \equiv \pm \bH(\ga) \modp$ and $\bH_{p'}(\laOw) \equiv \pm \bH(\ga) \modp$. If $\bH(\laiw)=2\bH(\laOw)$, and since $ \bH(\ga)$ is invertible $\modp$, this implies $1 \equiv \pm 2 \modp$. This can only happen if $p=3$, in which case $1 \equiv  -2 \modp$. 

We therefore suppose that $p=3$, and we first suppose that $e_i \geq i+p$. Looking at the case $w=1$ in the proof of Theorem \ref{bars}, we see that, writing $\{1, \, 2\} \setminus \{i\}=\{j\}$, we have
$$\frac{\bH(\la_i^{(1)})}{\bH(\ga)}=3(3+e_i)(3+e_i-j) \; \; \mbox{and} \; \; \frac{\bH(\la_0^{(1)})}{\bH(\ga)}=3(3+e_i)\frac{2}{i}.$$But since $e_i \geq i+p$, we have $3+e_i-j \geq 6+i-j \in \{ 5, \, 7 \}$, so that $3+e_i-j > \frac{2}{i}$ and $\bH(\la_i^{(1)}) > 2 \bH(\la_0^{(1)})$. Since, for each $2 \leq k \leq w$, we have $ \frac{\bH(\la_i^{(k)})}{\bH(\la_i^{(k-1)})} > \frac{\bH(\la_0^{(k)})}{\bH(\la_0^{(k-1)})}$, we finally obtain that, if $p=3$ and $e_i \geq i+p$, then $\bH(\laiw)>2 \bH(\laOw)$. Thus, in this case, we get, as above, $\psi_i, \, \psi_0 \in \Irr_0(B)$ such that $\psi_0(1) \neq \psi_i(1)$.

The last case to study is thus when $p=3$ and $e_i=i$. In this case, we have $\ga=(i)$, $\laiw=(pw+i)$ and $\laOw=(pw,\, i)$. The bar lengths in $\laiw$ are $\{ 1, \, 2, \, \ldots , \, pw+i \}$, while those in $\laOw$ are $\{ 1, \, \ldots , \, pw-(i+1), \, pw-(i-1), \, \ldots , \, pw, \, pw+i, \, 1, \, \ldots , \, i\}$. We thus have
$$\frac{\bH(\laiw)}{\bH(\laOw)}= \left\{ \begin{array}{ll} pw-1 & \mbox{if } i=1, \\ \frac{(pw+1)(pw-2)}{2}  & \mbox{if } i=2. \end{array} \right. $$Since $p=3$ and $w \geq 2$, we obtain in both cases that $\bH(\laiw)>2 \bH(\laOw)$. Thus, in this case also, we get $\psi_i, \, \psi_0 \in \Irr_0(B)$ such that $\psi_0(1) \neq \psi_i(1)$.

Finally, we have shown that, if there exists a unique $i$ such that $X_i \neq \emptyset$, then we can find $\psi_i, \, \psi_0 \in \Irr_0(B)$ such that $\psi_0(1) \neq \psi_i(1)$. This ends the proof.

\end{proof}

\begin{corollary}\label{finalresult}
Let $n \geq 4$ be any integer and $p$ be a prime. If $B$ is a $p$-block of $2.\A_n$ all of whose height zero characters have the same degree, then $B$ is nilpotent.

\end{corollary}
\begin{proof}
As we mentionned in Section \ref{covering}, if $p=2$ or if $p$ is odd and $B$ is a non-faithful block, then $B$ must have cyclic defect groups. Since Brauer's Height Zero Conjecture is obvious in this case (as all the characters are linear), \cite[Theorem 4.1]{Malle-Navarro} implies that $B$ is nilpotent.

If $p$ is odd and $B$ is a spin block, then, by Theorem \ref{mainresult}, $B$ must have abelian defect groups. Since Brauer's Height Zero Conjecture holds for $2.\A_n$ for $p$ odd (see \cite{Olsson-blocks}), we deduce from \cite[Theorem 4.1]{Malle-Navarro} that $B$ is nilpotent.
\end{proof}


\bibliographystyle{plain}
\bibliography{referencesJB}

\end{document}